\documentclass[11pt,a4paper]{article}

\usepackage{amsmath,amssymb,amsfonts,amsthm,latexsym,tikz}
\usepackage{graphicx, bm, bbm, yfonts, colortbl}
\usepackage[citecolor=blue,urlcolor=blue,hyperindex,breaklinks]{hyperref}
\usepackage{subcaption,float}
\usepackage{MnSymbol}

\newcommand {\bC}{\mathbb{C}}

\newcommand {\IN}{\mathbb{N}}

\newcommand {\IR}{\mathbb{R}}
\newcommand{\R}{\mathbb{R}}

\newcommand{\re}{\mathbb{R}}
\newcommand{\rek}{\mathbb{R}^{k}}

\newcommand {\B}{\mbox{${\mathcal B}$}}
\newcommand{\cB}{\mathcal B}
\newcommand{\cC}{\mathcal {C}}

\newcommand {\F}{\mbox{${\mathcal F}$}}
\newcommand{\cF}{\mathcal F}

\newcommand{\tf}{\mathcal{F}}

\newcommand{\cI}{\mathcal I}
\newcommand{\cJ}{\mathcal J}

\newcommand{\cL}{\mathcal L}

\newcommand {\cO}{\mathcal O}

\newcommand{\cX}{\mathcal X}

\newcommand{\cs}{\mathcal{S}}

\newcommand{\Lam}{\Lambda}

\newcommand{\half}{\frac{1}{2}}
\newcommand{\beq}{\begin{equation}}
\newcommand{\eeq}{\end{equation}}
\newcommand{\beqn}{\begin{equation*}}
\newcommand{\eeqn}{\end{equation*}}
\newcommand{\txind}{(t,x)\in [0,T]\times D}

\newtheorem{stat}{Statement}[section]
  \newtheorem{prop}[stat]{Proposition}
  
  \newtheorem{thm}[stat]{Theorem}
  \newtheorem{lemma}[stat]{Lemma}
  \newtheorem{remark}[stat]{Remark}
  \newtheorem{def1}[stat]{Definition}

\numberwithin{equation}{section}
\numberwithin{subsection}{section}

\begin{document}


{	\centering 
	
	
	
	
	
	
	{ \scshape \bf \large Sample path properties of parabolic SPDEs
with non constant coefficients}\footnote[1]{Version October 30, 2024



\noindent{\bf Statements and Declarations.} The first author was supported in part by the Swiss National Foundation for Scientific Research.
The second author was supported in part by the grant PDI2020-118339GB-100 from the Ministerio de Ciencia e Innovaci\'on, Spain, and by 2021 SGR 00299, Generalitat de Catalunya.

There are no competing interests.

} 
	
	
	
	
	
	
	

	\vspace{0.5 in} 
	
	{\scshape {\large Robert C.~Dalang}\\ \'Ecole Polytechnique F\'ed\'erale de Lausanne} 
	\vspace{0.5\baselineskip}
	
		{\scshape and}
		\vspace{0.5\baselineskip}
			
	{\scshape{\large Marta Sanz-Sol\'e}\\ University of Barcelona} 
	
	\vspace{0.5 in} 
	
   {\em In memory of Giuseppe Da Prato}

	
	
	
	
	
	



\thispagestyle{empty}

~




}


\begin{abstract}
We consider an SPDE driven by a parabolic second order partial differential operator with a nonlinear random external forcing defined by a Gaussian noise that is white in time and has a spatially homogeneous covariance. We prove existence and uniqueness of a random field solution to this SPDE. Our main result concerns the space-time sample path regularity of its solution.
\end{abstract}

\section{Introduction}
\label{s-1}

In this paper, we consider second order partial differential operators of the type
\beq
\label{5.300}
\cL = \frac{\partial}{\partial t} - \sum_{i,j=1}^k
 a_{i,j}(t,x)\frac{\partial^2}{\partial x_i \partial x_j} + \sum_{i=1}^k b_i(t,x)\frac{\partial}{\partial x_i} + c(t,x),
 \eeq
on $\re_+ \times D$, where $D \subset \rek$, $k\ge 1$, is a bounded or unbounded domain with smooth boundary, and $a_{i,j}(t,x)$, $b_i(t,x)$ and $c(t,x)$ are functions such that the fundamental solution (or the Green's function) of the PDE $\cL u=0$ (with boundary conditions if $D$ has boundaries) is a function. Fix $T > 0$. We are interested in the non-homogeneous {\em stochastic PDE} (SPDE)
\beq
\label{5.1}
\cL u(t,x) = \sigma(t,x,u(t,x))\, \dot W(t,x) + b(t,x,u(t,x)),
\eeq
$t\in\,]0,T]$, $x\in D$, with given initial conditions, and, if $D$ has boundaries, with boundary conditions. Here $\sigma$ and $b$ denote real-valued functions and $\dot W$ is a Gaussian noise that is white in time and with a spatially homogeneous covariance.
We prove existence and uniqueness of random field solutions and optimal results on the regularity of their sample paths.

In this context, the formulation of the SPDE \eqref{5.1} relies on the theory of stochastic integration developed first in \cite{dalang-1999} and on further extensions (see e.g.~\cite{dalang-quer-s-2011} and references herein). The regularity of the sample paths is proved using the factorization method of Da Prato, Kwapie\'n and Zabczyk introduced in \cite{dp-k-z-1987}. In this way, we extend and improve results of \cite{ss-v2003} and \cite{ss-s-2002}. Similar questions have been addressed in \cite{veraar-2010} in the setting of SPDEs with solutions in Banach spaces.

This article consists of four sections. In Section \ref{s2}, we fix the setting, the definitions and the hypotheses, and recall the properties of the noise $\dot W$ and of the stochastic integral with respect to $\dot W$. In Section \ref{s3}, we prove an existence and uniqueness theorem for random field solutions. The optimal H\"older-continuity, jointly in time and in space, of the random field solutions is proved in Section \ref{ch6-s1}.

Aiming to be concise and pedagogical, and to avoid too many technicalities, we opt for a presentation in which the hypotheses could be in part relaxed. In particular, concerning the noise, we assume that the spatial covariance is absolutely continuous with respect to Lebesgue measure, and the fundamental solution (or Green's function) $\Gamma(t,x;s,y)$ relative to $\cL$ is a function that is dominated by a homogeneous function $S(t-s, x-y)$ (see hypotheses $({\bf H_\Gamma})$ and $({\bf H_\Lambda})$).
 Although these hypotheses cover many important instances studied in the literature, such as Riesz, Bessel and fractional-type covariances, and (uniformly) parabolic operators $\cL$, Theorem \ref{ch5-t2} also holds under weaker conditions (see Remark \ref{s2-rem-1}).
 For the study of the regularity of the solutions via the factorization method, the positivity and semigroup property of $\Gamma(t,x;s,y)$, as well as  Gaussian bounds such as \eqref{gauss}, are indispensable. These conditions are ensured by \eqref{parabolic201} or by variants of this condition,
 such as for instance Petrovski\u{\i}-parabolicity, described in \cite[Chapter I]{ez}. Hence,
 it is unlikely that the assumptions of Theorem \ref{ch5-t30} can be significantly weakened.


\section{Formulating the problem}
\label{s2}

Consider the SPDE \eqref{5.1} with partial differential operator $\cL$ given in \eqref{5.300}, defined on $[0,T]\times D$, where $T>0$ and $D \subset \rek$, $k\ge 1$, are as in Section \ref{s-1}. 
The coefficients $\sigma$ and $b$ are defined on $[0,T]\times D\times \re$ and take values in $\re$.
Concerning $\dot W$, we consider a Gaussian noise that is white in time and spatially homogeneous. More specifically,
on a complete probability space $(\Omega, \mathcal{F},P)$, we define a family of mean zero Gaussian random variables
$\dot W = \{W(\phi),\, \phi \in \mathcal{C}_0^\infty (\mathbb{R}^{k+1})\}$ with covariance functional of the form
\begin{equation}
\label{ch2-1}
E(W(\phi_1) W(\phi_2)) = \int_0^\infty dt\, \langle \Lam, \phi_1(t)\star \tilde\phi_2(t) \rangle,
\end{equation}
where ``$\star$" denotes convolution in the spatial variable and $\tilde \phi (t,x):= \phi(t,-x)$.
In the above, $\Lam$ is a non-negative definite and symmetric distribution in  $\mathcal{D}^\prime(\rek)$.
It is therefore the Fourier transform of a non-negative, tempered, symmetric measure $\mu$ on $\mathbb{R}^{k}$, called the {\em spectral measure}, that is,
\beqn
   \langle \Lambda, \phi\rangle =\int_{\rek} \tf \phi(\xi)\, \mu(d\xi),
\eeqn
for all $\phi$ belonging to the space $\cs(\rek)$ of rapidly decreasing $\cC^\infty$-functions. Here,
\beqn
   \tf \phi(\xi) = \int_{\rek} e^{-i \xi x}\, \phi(x)\, dx.
\eeqn

Using elementary
properties of the Fourier transform, we see that
\beq
\label{ch2-2a}
  E(W(\phi_1) W(\phi_2)) = \int_0^\infty dt \int_{\mathbb{R}^{k}}  \mu(d\xi)\, \mathcal{F} \phi_1(t)(\xi)\, \overline{\mathcal{F} \phi_2(t)(\xi)} ,
\eeq
where $\overline z$ denotes the complex conjugate of $z \in \bC$.
In the particular case where $\Lambda$ is a measure on $\rek$ that is absolutely continuous with respect to Lebesgue measure with (necessarily symmetric) density function $f$, we have
\beq
\label{idcov}
\int_{\mathbb{R}^{k}}  \mu(d\xi)\, \mathcal{F} \phi_1(t)(\xi)\, \overline{\mathcal{F} \phi_2(t)(\xi)}
=\int_{\mathbb{R}^{k}} dx \int_{\mathbb{R}^{k}} dy\, \varphi_1(t,x) f(x-y)\varphi_2(t,y),
\eeq
and formula \eqref{ch2-1} (and \eqref{ch2-2a}) become
\beq
\label{ch2-2b}
  E(W(\phi_1) W(\phi_2)) = \int_0^\infty dt \int_{\mathbb{R}^{k}} dx \int_{\mathbb{R}^{k}} dy\, \varphi_1(t,x) f(x-y)\varphi_2(t,y).
\eeq

Throughout this article, we will assume the following set of hypotheses.
\vskip 12pt

\noindent {\em Assumptions on the fundamental solution}
\medskip

\noindent $({\bf H_{\Gamma}})$  The functions $a_{i,j}(t,x)$, $b_i(t,x)$, $c(t,x)$ from $[0,T]\times D$ into $\re$ are $\cB_{[0,T]}\times \cB_D$-measurable and such that the fundamental solution $\Gamma$ of $\cL$ satisfies the following:
\begin{description}
\item{(a)} $\Gamma$ is a function from $$\{(t,x,s,y)\in [0,T]\times D\times [0,T]\times D: 0\le s<t\le T\}$$ into $\re$, jointly Borel measurable, that is, $\cB_{[0,T]}\times \cB_{D}\times \cB_{[0,T]}\times \cB_{D}$-measurable.
\item{(b)} There is a Borel function $S: [0,T]\times \rek\longrightarrow \re_+$ satisfying $S(t,\ast)\in L^1(\rek)$ for all $t\in[0,T]$, and $ C < \infty$ such that for all $(t, x), (s, y) \in [0, T] \times D$,
\beq
\label{HGamma}
 \vert\Gamma(t,x;s,y)\vert \le C S(t-s,x-y).
 \eeq

 Further,
\beq\label{parabolic200}
\int_0^T ds \int_{\rek} \mu(d\xi)\, |\tf S(s)(\xi)|^2 < \infty
\eeq
and
\beq \label{parabolic200-bis}
   \int_0^T ds \int_{\rek} dy\,S(s,y) < \infty.
\eeq
\end{description}
\medskip

Notice that \eqref{parabolic200} is a joint condition on $\cL$ (through $\Gamma$) and the noise $\dot W$ (through its spectral measure $\mu)$.

An important class of operators $\cL$ in \eqref{5.300} are {\em second-order uniformly parabolic operators}. They satisfy
\beq
\label{parabolic201}
\sum_{i,j=1}^k a_{i,j}(t,x) y_i y_j\ge \rho\,  |y|^2,
\eeq
 for all $y\in D$, with $\rho>0$ (see \cite{ez}).
 In this case, it is known that under suitable regularity conditions on the functions $a_{i,j}(t,x), b_i(t,x), c(t,x)$, the fundamental solution $\Gamma$ has the following properties:
\begin{description}
 \item{(i)} $\Gamma(t,x;s,y)>0$;
 \item{(ii)}\ {\em Semigroup property}: for any $0<s<r<t$, $x,y\in D$,
 \beq
 \label{convol}
  \Gamma(t,x;s,y) = \int_{D} dz\  \Gamma(t,x;r,z)\Gamma(r,z;s,y);
  \eeq
 \item {(iii)}\ {\em Gaussian bounds:}
 \beq
 \label{gauss}
 \vert \partial_x^\mu \partial_t^\nu\Gamma (t,x;s,y)\vert \le C(t-s)^{-\frac{k+|\mu|+2\nu}{2}} \exp\left[-c\, \frac{|x-y|^2}{t-s}\right],
 \eeq
 where $\mu=(\mu_1, \ldots, \mu_k)\in{\IN}^d$, $\nu\in \IN$, and $|\mu|+2\nu\le 2$, with $|\mu|= \sum_{j=1}^k \mu_j$
 (see, e.g.~\cite[p. 658]{aronson-1968}, \cite[Theorem VI.6, p. 193]{ez}, \cite[Theorem 1.1, p. 183]{ei70} or \cite[Chapter 9, Theorem 8]{friedman64}).
 \end{description}

In the particular case $|\mu|+\nu=0$, \eqref{gauss} reads
 \beq
 \label{gauss-simple}
\vert\Gamma (t,x;s,y)\vert = \Gamma (t,x;s,y)\le C(t-s)^{-\frac{k}{2}} \exp\left[-c\, \frac{|x-y|^2}{t-s}\right].
\eeq
This yields property \eqref{HGamma} with
\beq
\label{relevant}
S(t,x): = t^{-\frac{k}{2}} \exp\left[-c\, \frac{|x|^2}{t}\right],
\eeq
which clearly satisfies $S(t,\ast)\in L^1(D)$.
From  the properties of the Gaussian distribution, we deduce
that
\beq
\label{boundgauss}
\int_{\rek} S(s,y)\, dy\le C < \infty,
\eeq
with a constant $C$ that does not depend on $s\in[0,T]$ and thus, \eqref{parabolic200-bis} is satisfied.

\vskip 12pt

\noindent{\em Assumptions on the covariance measure}
\medskip

\noindent $({\bf H_{\Lambda}})$\ The covariance measure $\Lambda$ is such that
$\Lambda(dx) = f(x)\, dx$, $\mu(d\xi)= (\tf f)(\xi)\, d\xi$, where $f : \rek\rightarrow \re_+$ is a real-valued function which satisfies {\em one of} the following two conditions:
\begin{description}
\item {~~(i)} $f(x)<\infty$ if and only if $x\ne 0$;
\item{or} ~
\item {~~(ii)}  $\tf f \in L^\infty(\rek)$, and $f(x)< \infty$ when $x\ne 0$.
\end{description}




\medskip

\noindent{\em Assumptions on the functions $\sigma$ and $b$}
\medskip

\noindent${\bf(H_L)}$\ The functions $\sigma$ and $b$ are defined on $[0,T]\times D\times \re$ with values in $\re$ and are jointly measurable, that is, $\cB_{[0,T]}\times \cB_{D}\times \cB_{\re}$-measurable, and have the following two properties:
\smallskip

\noindent{\em Uniform Lipschitz continuity.} There exists  $C:= C(T,D) < \infty$
 such that for all $(s,y)\in [0,T]\times D$ and $z_1, z_2\in \re$,
 \beqn
|\sigma(s,y,z_1) - \sigma(s,y,z_2)| + |b(s,y,z_1) - b(s,y,z_2)| \le C\,  |z_1-z_2|.
\eeqn
\smallskip

\noindent{\em Uniform linear growth.}\ There exists a constant $\bar c:=\bar c(T,D) < \infty$ such that for all $(s,y)\in [0,T]\times D$ and $z\in \re$,
\smallskip

\beqn
|\sigma(s,y,z)| + |b(s,y,z)| \le \bar c\, (1 + |z|).
\eeqn
\vskip 12pt

\noindent{\em Assumptions on the initial conditions}
\medskip

Let $[0,T]\times D \ni (t,x)\mapsto I_0(t,x)$ be the solution to the homogeneous PDE $\cL I_0(t, x) = 0$ with the same initial and boundary conditions as in \eqref{5.1}. We assume that:
\smallskip

\noindent $({\bf H_I})$\ The function  $[0,T]\times D \ni (t,x)\mapsto I_0(t,x)$ is Borel and bounded.
\medskip

For example, in the case where $D=\rek$, $\cL$ has constant coefficients and the initial value of the SPDE \eqref{5.1} is $u(0,x):=u_0(x)$, then $\Gamma(t, x; s, y) = \Gamma(t-s, x-y)$ and
\beq
\label{5.5}
I_0(t,x) = (\Gamma(t)\star u_0)(x).
\eeq


\vskip 12pt

\noindent{\em Examples of covariance measures satisfying $({\bf H_\Lambda})$}
\medskip

We describe some examples of covariances that have been extensively used in the study of SPDEs driven by spatially correlated Gaussian noises.
\medskip

\noindent{\bf Riesz kernels}. \ For $\beta\in\,]0,k[$, let $r_\beta: \rek \rightarrow [0,\infty]$ be defined by  $r_\beta(x) = \vert x \vert^{-\beta}$ for $x\in\rek\setminus \{0\}$, and $r_\beta(0) = +\infty$. The function $r_\beta$  
is related to the fundamental solution of the fractional power of the Laplacian: $(- \Delta)^{\frac{k - \beta}{2}}$ (see e.g. \cite[Chapter V, \S 1]{stein-1970}). Up to a multiplicative constant, $r_\beta$ is the Fourier transform 
of $\mu_\beta(d\xi) = r_{k - \beta}(\xi)\, d\xi$:
\beqn
 c_{k,\beta}\, \tf r_{k - \beta}(\xi) =    r_{\beta}(\xi),\quad \xi\in\rek
\eeqn
(see \cite[Chapter V, \S 1, Lemma 2 (a)]{stein-1970}).

Observe that for $\eta\in\, ]0,1]$,
\beqn
\int_{\rek} \frac{\mu_\beta(d\xi)}{(1+|\xi|^2)^\eta} < \infty\ {\text{if and only if}}\ \beta\in\,]0, k \wedge (2\eta)[,
\eeqn
because up to a multiplicative constant, $\tf\mu_\beta(d\xi)= r_{\beta}(\xi) d\xi$.

This remark is relevant for checking conditions \eqref{parabolic200} and \eqref{5.311}. For example, if $\beta\in\,]0,2\wedge k[$ and $S(t,x)$ is as in \eqref{relevant}, then \eqref{parabolic200} holds.
\medskip

\noindent{\bf Bessel kernels}.\
For $\alpha>0$, let
\beqn
 f_\alpha(x) =
\int_0^\infty w^{(\alpha -k-2)/2} e^{-w-(\vert x \vert^2)/(4w)}\, dw, \quad {\text{if}}\ x\in \rek\setminus\{0\},
\eeqn
and  $f_\alpha(0) = +\infty$ if $\alpha\in\,]0,k]$, and $f_\alpha(0) = c(k,\alpha)$ if $\alpha>k$, where $c(k,\alpha)\in ]0,\infty[$.
The function $f_\alpha$ is 
related to the fundamental solution of the partial differential operator $(I-\Delta)^{-\frac{\alpha}{2}}$ (see \cite[Chap. V, \S 4]{stein-1970}, or  \cite[p. 152]{folland}).

 The Fourier transform $\tf f_\alpha$ is the function
 \beqn
 \tf f_\alpha(\xi) = \Gamma_E\left(\frac{\alpha}{2}\right)(4\pi)^{k/2} \tf B_\alpha(\xi) = \Gamma_E\left(\frac{\alpha}{2}\right)(4\pi)^{k/2} (1+|\xi|^2)^{-\frac{\alpha}{2}},
 \eeqn
where $\Gamma_E$ denotes the Euler Gamma function (see \cite[Chap. V, \S 3, Proposition 2]{stein-1970})). Thus, the spectral measure $\tilde\mu_\alpha$ satisfies the following:
 for any $\eta\in ]0,1]$,
\beqn
\int_{\rek} \frac{\tilde\mu_\alpha(d\xi)}{(1+|\xi|^2)^\eta} < \infty\ {\text{if and only if}}\ \alpha>k-2\eta.
\eeqn
\smallskip

 \noindent{\bf Fractional kernels}.\ Let
    $H=(H_1, \ldots,H_k)\in\, \left]\frac{1}{2},1\right[^k$, with $\sum_{j=1}^k H_j > k-1$. Define
\beqn
f_H(x) = c_{k,H}\prod_{j=1}^k |x_j|^{2H_j-2},\     c_{k,H}=  \prod_{j=1}^k H_j(2H_j-1),\ x=(x_j)_j,
\eeqn
when  $\prod_{j=1}^k \xi_j\ne 0$, and $f_H(x)=+\infty$, otherwise.

The spectral measure is
\beqn
\mu_H(d\xi) = \tilde c_{k,H}  \prod_{j=1}^k |\xi_j|^{1-2H_j} d\xi\, ,
\eeqn
which implies that, for any $\eta\in ]0,1]$,
\beqn
\int_{\rek} \frac{\bar\mu_H(d\xi)}{(1+|\xi|)^\eta} < \infty\ {\text{if and only if}}\ \sum_{j=1}^k H_j>k-\eta.
\eeqn
Detailed computations that establish this equivalence can be found for instance in the extended version of \cite{millet-SanzSole-2021} (see arXiv:1911.03148v2).
\vskip 12pt

\noindent{\em The stochastic integral with respect to $\dot W$}
\medskip

We now recall the notion of stochastic integral with respect to the Gaussian noise $\dot W$ that will be used in the random field formulation of the SPDE \eqref{5.1} (see Definition \ref{ch5-d1} below). For this, we will first recall how we obtain a cylindrical Wiener process from $\dot W$.

On $\mathcal{C}_0^\infty (\rek)$ we consider the semi-inner product
\beq
\label{sec2-1}
\langle \varphi_1,\varphi_2\rangle_U = \int_{\mathbb{R}^{k}}  \mu(d\xi)\, \mathcal{F} \varphi_1(\xi)\, \overline{\mathcal{F} \varphi_2(\xi)}
\eeq
and the associated semi-norm $\Vert \cdot\Vert_U$, and  we identify elements $\varphi_1$, $\varphi_2$ such that $\Vert \varphi_1-\varphi_2\Vert_U=0$. We denote by $U$ the completion of the resulting space endowed with the norm $\Vert\cdot\Vert_U$, and define $U_T=L^2([0,T]; U)$.
The space $\mathcal{C}_0^\infty([0,T]\times \rek)$ is dense in $U_T$ for the norm
$$\Vert g\Vert_{U_T}= \left(\int_0^T \Vert g(s)\Vert_U^2\, ds\right)^{\half},\quad g \in U_T.$$
Using this fact and the Itô isometry, we can extend the process $\dot W$ to a Gaussian process $W = (W(\varphi),\, \varphi\in U_T)$ and, setting $W_t(\varphi) = W(1_{[0,t]}(\cdot)\varphi(\ast))$, we see that 
the process $W = (W_t(\varphi),\, t\in[0,T], \varphi\in U)$ is a {\em standard cylindrical Wiener process} (that we still  denote by $W$). In particular, if $(e_j)_{j\ge 1}$ is a CONS of the Hilbert space $U$, then $(W_t(e_j),\, 0\le t\le T)$, $j\ge 1$, defines a sequence of independent Brownian motions.

For its further use, we quote the following result from \cite[Theorem 3.5]{BGP}.

\begin{prop}
\label{ch2-incluU}
Assume that $\mu$ is absolutely continuous with respect to Lebesgue measure. Then the set
\beqn
\mathcal{S}^\prime_\mu(\IR^k)  = \left\{ \Gamma \in \mathcal{S}^\prime(\IR^k): \text{$\tf\Gamma$ is a function and } \int_{\rek} \mu(d\xi)\, \vert \tf \Gamma(\xi)\vert^2 < \infty\right\}
\eeqn
is included in $U$, and for such $\Gamma$,
\beq
\label{norma}
\Vert \Gamma\Vert_ U^2 = \int_{\rek} \mu(d\xi) \, \vert \tf \Gamma(\xi)\vert^2.
\eeq
\end{prop}

In the sequel, we consider the probability space $(\Omega, \tf, P)$ on which the process $W$ is defined, along with the filtration $(\tf_t)_{t\ge 0}$ generated by $(W_t(\varphi),\, t\in[0,T], \varphi\in U)$. 
\begin{def1}
\label{si-d1}
For a  jointly measurable and adapted stochastic process $H\in L^2(\Omega\times [0,T]; U)$, the stochastic integral of $H$ with respect to the standard cylindrical Wiener process $W$ is
\beq
\label{si-1}
\int_0^T H_s\, dW_s:= \sum_{j=1}^\infty \int_0^T\langle H_s,e_j\rangle_U\, dW_s(e_j).
\eeq
\end{def1}
\smallskip

The stochastic integral will also be denoted by $\int_0^T \int_{\rek}H(s,y)\, W(ds,dy)$. It does not depend on the particular choice of $(e_j)_{j\ge 1}$.

Observe that the right-hand side of \eqref{si-1} is a series of one-dimensional It\^o integrals with respect to the independent Brownian motions $W_s(e_j)$, $j\ge 1$. It converges in $L^2(\Omega)$. Moreover, we have
the isometry property
\beq
\label{si-2}
E\left[\left(\int_0^T H_s\, dW_s\right)^2\right] = E\left(\int_0^T \Vert H_s\Vert^2_U\, ds\right) = E\left(\Vert H\Vert_{U_T}^2\right).
\eeq
The indefinite integral process $\left(\int_0^t H_s\, dW_s,\, t\in [0,T]\right)$, defined in the usual way, is a martingale with respect to the filtration
$(\tf_ t, \, t\in [0,T])$. By the independence of $W_t(e_j)$, $j\ge 1$, its quadratic variation is the process
\beq
\label{ch3-qv}
\int_0^t \Vert H_s\Vert^2_U\, ds, \quad t\in [0,T].
\eeq
\smallskip

\noindent{\em Examples of integrands}
\medskip

In the random field formulation of SPDEs, we will often find integrands $H$ of the type $G_{t,x}(s,y) = \Gamma(t,x;s,y) Z(s,y) \, 1_D(y)$,
$0\le s<t\le T$, $y\in \rek$, with fixed $(t, x) \in [0, T] \times D$, where $(Z(s,y),\,  (s,y)\in[0,T]\times D)$ is a jointly measurable and adapted process satisfying $\sup_{(s,y)\in[0,T]\times D}E(Z^2(s,y))<\infty$.
Observe that this condition, along with $({\bf H_{\Gamma}})$, implies that $G_{t,x}(s,\ast)$ is the density of a measure on $\rek$ with bounded total variation.

Assuming $({\bf H_{\Gamma}})$ and $({\bf H_{\Lambda}})$, we next show that $E\left(\Vert G_{t,x}\Vert_{U_T}^2\right)<\infty$ and consequently, the stochastic integral
\beq
\label{si-ex-1}
\int_0^T \int_{D} \Gamma(t,x;s,y) Z(s,y) \, W(ds,dy)
\eeq
is well-defined.

The proof relies on the following lemma quoted from \cite[Theorem 5.2, p. 481]{k-x-2009} (see also \cite[Corollary 3.4, p 426]{k-f-2013}):
\begin{lemma}
\label{l-1}
Assume $({\bf H_{\Lambda}})$. There is a constant $C_k$, depending on $k$ only, with the following property. Let $\Phi$ and $\Psi$ be signed measures on $\rek$ with bounded total variation. Then
\beq
\label{l-1-1}
\int_{\rek} \Phi(dx) \int_{\rek} \Psi(dy)\, f(x-y) = C_k \int_{\rek} \mu(d\xi)\, \tf \Phi(\xi) \overline{\tf \Psi(\xi)} .
\eeq
\end{lemma}
\medskip

Applying Proposition \ref{ch2-incluU} and Lemma \ref{l-1} to $\Phi(\ast)=\Psi(\ast)$, where, for fixed $(t, x; s)$, $\Phi(\ast)$ is the absolutely continuous (with respect to Lebesgue measure) signed measure with density $G_{t,x}(s,\ast)$, we have
\begin{align}
\label{comp-norm}
E\left(\Vert G_{t,x}\Vert_{U_T}^2\right)
&=C E\left(\int_0^t ds \int_D dy \int_D dz\, G_{t,x}(s,y) f(y-z)\, G_{t,x}(s,z)\right)\notag\\
&\le C E\left(\int_0^t ds \int_D dy \int_D dz\, |\Gamma(t,x;s,y)|\, |Z(s,y)|\, f(y-z)\right.\notag\\
 &\left.\qquad \qquad \qquad\qquad\qquad\quad\times\, |\Gamma(t,x;s,z)|\, |Z(s,z)|\right).
\end{align}
Using assumption $({\bf H_{\Gamma}})$(b), this expression is bounded above by
\begin{align*}
&\tilde C E\Big(\int_0^t ds \int_D dy \int_D dz\,  S(t-s,x-y)\, |Z(s,y)|\, f(y-z)\\
&\qquad \qquad \qquad\qquad\qquad\times S(t-s,x-z)\, |Z(s,z)|\Big)\\
&\qquad \le \tilde C \sup_{(s,y)\in[0,T]\times D}E(Z^2(s,y)) \int_0^t ds \int_{\rek} \mu(d\xi)\,|\tf S(s)(\xi)|^2.
\end{align*}
Therefore,
\beq
\label{2-3}
E\left(\Vert G_{t,x}\Vert_{U_T}^2\right)\le \tilde C  \sup_{(s,y)\in[0,T]\times D}E(Z^2(s,y)) \int_0^t ds \int_{\rek}\mu(d\xi)\, |\tf S(s)(\xi)|^2 <\infty,
\eeq
by \eqref{parabolic200}. This shoes that the stochastic integral \eqref{si-ex-1} is well-defined.

\begin{remark}
\label{s2-rem-1}
One can show that $E\left(\Vert G_{t,x}\Vert_{U_T}^2\right) < \infty$ under weaker assumptions than $({\bf H_{\Gamma}})$ and $({\bf H_{\Lambda}})$, and then the stochastic integral \eqref{si-ex-1} remains well defined. We mention two possibilities:
\medskip

\noindent   $({\bf H_{\Lambda,\Gamma}}(1))$.\ $({\bf H_{\Lambda}})$ holds and,  for any $t\in [0,T]$, $\Gamma(t)$ is a measure with finite total variation $\Vert \Gamma(t)\Vert$, that satisfies
\beq
\label{hgammatype}
\int_0^T dt \int_{\rek}\mu(d\xi)\, \big\vert \tf \Vert \Gamma(t)\Vert (\xi)\big\vert^2 < \infty.
\eeq
\smallskip

\noindent   $({\bf H_{\Lambda,\Gamma}}(2))$.\ The covariance measure $\Lambda$ is a nonnegative  non-negative definite symmetric distribution. The fundamental solution $\Gamma$ is such that for any $t\in [0,T]$, $\Gamma(t)$ is a measure and its total variation $\Vert \Gamma(t)\Vert$ belongs to the set $\cO^\prime_{\cC}(\rek)$ of distributions with rapid decrease, and \eqref{hgammatype} holds.
\smallskip

Indeed, the sufficiency of $({\bf H_{\Lambda,\Gamma}}(1))$ follows from the arguments that come after Lemma \ref{l-1}, and that of $({\bf H_{\Lambda,\Gamma}}(2))$ can be checked as in \cite[Proposition 3.3]{nualartquer}.
\end{remark}

\medskip

\noindent{\em The notion of random field solution}
\vskip 12pt

\begin{def1}
\label{ch5-d1}
Assume $({\bf H_{\Gamma}})$ and $({\bf H_{\Lambda}})$. A jointly measurable and adapted real-valued random field $(u(t,x),\, (t,x)\in [0,T]\times D)$ is a {\em random field solution} of the SPDE \eqref{5.1} if, for all $(t,x)\in [0,T]\times D$,  the following identity is satisfied:
\begin{align}
\label{5.4}
u(t,x) & = I_0(t,x) + \int_0^t \int_{D} \Gamma(t,x;s,y) \sigma(s,y,u(s,y))\, W(ds,dy)\notag\\
&\qquad + \int_0^t ds  \int_{D}  dy\, \Gamma(t,x;s,y) b(s,y,u(s,y)), \ a.s.,
\end{align}
where $I_0(t,x)$ is the solution of the homogeneous PDE with the same initial (and boundary) conditions.
\end{def1}

\begin{remark}
\label{ch5-r1}
\begin{enumerate}
 \item Recall that the notation for the stochastic integral on the right-hand side of \eqref{5.4} means that we are integrating the stochastic process $(G_{t,x}(s,\ast):=\Gamma(t,x;s,\ast) \sigma(s,y,u(s,\ast))1_D(\ast), \, s \in[0,T])$ with respect to the standard cylindrical Wiener process $W$, in the sense of Definition \ref{si-d1}. Hence,  this process should be
 measurable, adapted and in $L^2(\Omega\times [0,T]; U)$.
Assuming  $({\bf H_{\Gamma}})$ and
$({\bf H_{\Lambda}})$, we have proved above that the latter condition is assured provided
\beqn
\sup_{(t,x)\in[0,T]\times D} E (|\sigma(t,x,u(t,x))|^2) < \infty.
\eeqn
\item Assuming $({\bf H_{\Gamma}})$ and that
\beqn
\sup_{(t,x)\in[0,T]\times D} E (|b(t,x,u(t,x))|) < \infty,
\eeqn
the pathwise integral on the right-hand side of \eqref{5.4} is well-defined. Indeed, using \eqref{parabolic200-bis}, we have
\begin{align}
\label{convol-2}
&E\left(\int_0^t ds \int_D dy\, |\Gamma(t,x;s,y)|\, |b(s,y,u(s,y))|\right)\notag\\
&\qquad\qquad \le CE\left(\int_0^t \int_{D} dy\, S(t-s,x-y) \, |b(s,y,u(s,y)|\right)\notag\\
&\qquad\qquad \le C \sup_{(t,x)\in[0,T]\times D} E (|b(t,x,u(t,x))|)\int_0^tds \int_{\rek}dy\,  S(s,y) < \infty.
\end{align}
\end{enumerate}
\end{remark}

\section{Existence and uniqueness of solutions}
\label{s3}

In this section, we prove a theorem on existence and uniqueness of a random field solution to the SPDE \eqref{5.1}.

\begin{thm}
\label{ch5-t2}
We assume $({\bf H_{\Lambda}})$,
$({\bf H_{\Gamma}})$, $({\bf H_L})$ and $({\bf H_I})$. Then the SPDE \eqref{5.1} with $\cL$ given by \eqref{5.300} has a random field solution
$(u(t,x),\, (t,x)\in[0,T]\times D)$ in the sense of Definition \ref{ch5-d1}, that is,
\begin{align}
\label{5.301}
u(t,x) & = I_0(t,x) + \int_0^t \int_{D} \Gamma(t,x;s,y) \sigma(s,y,u(s,y)) W(ds,dy)\notag\\
&\qquad + \int_0^t ds  \int_{D} dy\, \Gamma(t,x;s,y) b(s, y,u(s,y)), \ a.s.,
\end{align}
 for all $(t,x)\in [0,T]\times D$.

Moreover, for any $p\ge 1$,
\beq
\label{5.302}
\sup_{(t,x)\in [0,T]\times D}E\left(|u(t,x)|^p\right) < \infty,
\eeq
and this random field solution is unique (in the sense of versions) among random field solutions that satisfy \eqref{5.302} with $p=2$.
\end{thm}
\begin{proof}
We follow the scheme of the proof of \cite[Theorem 4.2.1]{d-ss-2024} for nonlinear SPDEs driven by space-time white noise, with the changes due to the more general covariance $\Lambda$.

Consider the Picard iteration scheme
$
 u^0(t,x)  = I_0(t,x),
$
 and for $n \geq 0$,
 \begin{align}\label{5.7}
    u^{n+1}(t,x) & =  u^0(t,x) + \int_0^t \int_{D} \Gamma(t,x;s,y) \sigma(s,y,u^n(s,y)) W(ds,dy)\notag\\
&\qquad + \int_0^t ds  \int_{D} dy\, \Gamma(t,x;s,y) b(s,y,u^n(s,y)), \quad a.s.
\end{align}
\smallskip

\noindent{\em Step 1}. We prove by induction on $n$ that for each $n\ge 0$, the process
\beqn
(u^n(t,x),\, (t,x)\in[0,T]\times D),\quad n\ge 1,
\eeqn
 is well-defined, jointly measurable (that is, has a jointly measurable version) and adapted, and for any $p\ge 2$,
\beq
\label{5.6n}
\sup_{(t,x)\in [0,T]\times D}E\left(|u^n(t,x)|^p\right) < \infty.
\eeq
These properties for $u^n$ imply that the stochastic and pathwise integral terms in \eqref{5.7} are well-defined, according to Definition \ref{si-d1} and Remark \ref{ch5-r1} point 2., respectively.

Indeed, denote by $g$ either the coefficient $\sigma$ or $b$. Assuming these properties for $u^n$, we infer that the map $(s,y,\omega)\mapsto g(s,y,u^n(s,y,\omega))$ from $([0,T]\times D\times \Omega), \cB_{[0,T]}\times \cB_{D}\times \tf)$ into $(\re,\cB_{\re})$ is measurable, since it is the composition of the measurable maps $(s,y,\omega)\mapsto(s,y,u^n(s,y,\omega))$ with $(s,y,z)\mapsto g(s,y,z)$. With similar arguments, we argue that the map $(s,y,\omega)\mapsto g(s,y,u^n(s,y,\omega))$ is adapted, that is, for fixed $s$, $(y, \omega) \mapsto g(s, y, u^n(s,y,\omega)$ is  measurable from
$(D\times \Omega, \cB_{D}\times \tf_s)$ into $(\re,\cB_{\re})$.



We now check that the stochastic integral term in \eqref{5.7} is well-defined. Observe that this integral is \eqref{si-ex-1} with $Z(s,y):= \sigma(s,y,u^n(s,y))$. By the uniform linear growth of $\sigma$ (hypothesis $({\bf H_L})$),
\beqn
\sup_{(s,y)\in[0,T]\times D} E\left(\sigma^2(s,y,u^n(s,y))\right) \le 2\bar c^2 \left(1+\sup_{(s,y)\in[0,T]\times D} E[|u^n(s,y)|^2]\right)<\infty
\eeqn
by \eqref{5.6n}. Hence, applying \eqref{2-3} with
\beqn
G_{t,x}(s,y):=\Gamma(t,x;s,y) \sigma(s,y,u^n(s,y))\, 1_D(y),
\eeqn
we finish the verification.

The pathwise integral term in \eqref{5.7} is also well-defined. Indeed, by \eqref{5.6n}, 
\beqn
\sup_{(s,y)\in[0,T]\times D} E\left(|b(s,y,u^n(s,y)|\right) \le \bar c \left(1+\sup_{(s,y)\in[0,T]\times D} E[|u^n(s,y)|]\right)<\infty.
\eeqn
Hence, from \eqref{gauss-simple} and \eqref{parabolic200-bis},  
we obtain the result.

We now start the induction for the proof of Step 1. Let $n=0$ and fix $p\ge 2$. Assumption $({\bf H_I})$ implies that $u^0=I_0$ satisfies the conditions of Step 1.

Assume that for some $n\ge 0$, the process $(u^n(t,x),\ (t,x)\in[0,T]\times D)$ is well-defined, jointly measurable and adapted, and satisfies \eqref{5.6n} for all $p\ge 2$. According to the discussion above, the process $u^{n+1}=(u^{n+1}(t,x),\, (t,x)\in[0,T])\times D)$ given in \eqref{5.7} is well-defined. We want to prove that $u^{n+1}$ is jointly measurable and adapted, and that it satisfies \eqref{5.6n} (with $n$ there replaced by $n+1$). Define
\begin{align*}
\cI^n(t,x) &:= \int_0^t \int_D \Gamma(t,x;s,y) \sigma(s,y,u^n(s,y))\, W(ds,dy),\\
\cJ^n(t,x) &:=  \int_0^t ds \int_D dy\, \Gamma(t,x;s,y) b(s,y,u^n(s,y)).
\end{align*}
Assuming $({\bf H_{\Lambda}})$ and $({\bf H_{\Gamma}})$, we can extend the proofs of Propositions 2.6.2 and 2.6.3 in \cite{d-ss-2024} to the setting of our theorem. In this way, we obtain the existence of jointly measurable and adapted modifications of $\cI^n(t,x)$ and $\cJ^n(t,x)$, respectively. With these modifications, we define a jointly measurable and adapted version of $u^{n+1}$ (which we denote also by $u^{n+1}$).

Let
\beq
\label{kappas}
\kappa_1(s) = \int_D dy \int_D dz\, S(s,x-y) f(y-z)S(s,x-z),\quad \kappa_2(s) = \int_D dy\, S(s,y).
\eeq
Since $S$ and $f$ are nonnegative, using \eqref{idcov}, we see that
\beq
\label{kappas-bis}
\kappa_1(s) \le  \int_{\rek}  \mu(d\xi)\,|\tf S(s)(\xi)|^2,\quad \kappa_2(s) \le \int_{\rek} dy\, S(s,y).
\eeq

Define
$G_{t,x}^n(s,y) = \Gamma(t,x;s,y)\sigma(s,y,u^n(s,y)) 1_D(y)$
and fix $p\ge 2$. By Burkholder's inequality, then H\"older's inequality, and because of Assumptions $({\bf H_\Gamma})$ and $({\bf H_L})$, we see that
\begin{align}
\label{3-10}
&E(|\cI^n(t,x)|^p)\notag\\
&\quad \le C_p E\left(\left\vert\int_0^t ds \int_{D} dy \int_{D} dz\, G_{t,x}(s,y) f(y-z)G_{t,x}(s,z)\right\vert^{\frac{p}{2}}\right)\notag\\
&\quad \le C_p E\left(\Big\vert\int_0^t ds \int_{D} dy \int_{D} dz\, S(t-s,x-y ) \vert\sigma(s,y,u^n(s,y))|\,f(y-z)\right.\notag\\
&\left.\qquad \qquad \times S(t-s,x-z ) \, |\sigma(s,y,u^n(s,z))|\,\Big\vert^{\frac{p}{2}}\right)\notag\\
&\quad \le \tilde C_p\left(\int_0^t ds\, \kappa_1(s)\right)^{\frac{p}{2}-1} \int_0^t ds\,  \kappa_1(t-s)\,\sup_{x\in D} \, (1+E(|u^n(s,x)|^p).
\end{align}
From this, we obtain
\begin{align}
\label{3-11}
E(|\cI^n(t,x)|^p)
&\le \bar C_p  \sup_{(t,x)\in [0,T]\times D}\left(1+E(|u^n(t,x)|^p)\right)\notag\\
&\qquad\qquad \qquad\qquad\times \left(\int_0^t ds  \int_{\rek} \mu(d\xi)\, |\tf S(s)(\xi)|^2\right)^{\frac{p}{2}} < \infty
\end{align}
by \eqref{5.6n}, \eqref{kappas-bis} and \eqref{parabolic200}.

Using H\"older's inequality and Assumptions $({\bf H_\Gamma})$ and $({\bf H_L})$, we obtain
\begin{align}
\label{3-12}
 E\left(\vert \cJ^n (t,x)\vert^p\right) 
 &\leq \bar c_p \left(\int_0^t ds\, \kappa_2(s)\right)^{p-1}\notag\\
 &\qquad\qquad\times \int_0^t ds\, \kappa_2(t-s) \, \sup_{x\in  D}\left(1+E(\vert u^n(s,x)\vert^p)\right) .
\end{align}
This implies
\beq
\label{3-13}
E(|\cJ^n(t,x)|^p)\le \bar c_p  \sup_{(t,x)\in [0,T]\times D}\left(1+E(|u^n(t,x)|^p\right)\left(\int_0^t ds\, \kappa_2(s)\right)^{\frac{p}{2}} < \infty
\eeq
by \eqref{5.6n}, \eqref{kappas-bis} and \eqref{parabolic200-bis}.

Adding \eqref{3-11} and \eqref{3-13}, we obtain
\beq
\label{step1-end}
 \sup_{(t,x)\in [0,T]\times D} E\left(|u^{n+1}(t,x)|^p\right)\le \tilde C_p\left( 1+ \sup_{(t,x)\in [0,T]\times D}\left(1+E(|u^n(t,x)|^p\right)\right) < \infty.
 \eeq
 This ends the proof of Step 1.
 \vskip 12pt

 \noindent{\em Step 2.} We show that the sequence of processes
$\left(u^n(t,x),\, (t,x)\in [0,T] \times D\right)$, $n\ge 0$,
converges in $L^p(\Omega)$ uniformly in $(t,x)\in [0,T] \times D$ to a process
\beqn
\left(u(t,x),\, (t,x)\in [0,T] \times D\right)
\eeqn
 that satisfies \eqref{5.301} and has a jointly measurable and adapted version.

Indeed, let
$M_n(t)= \sup_{y\in D} E\left(\left\vert u^{n+1}(t,y)-u^n(t,y)\right\vert^p\right)$, $n\ge 0$.
Using arguments similar to those used to deduce \eqref{3-10} and \eqref{3-12}, but applying the uniform Lipschitz continuity property of $\sigma$ and $b$ of Assumption $({\bf H_L})$ instead of the property of linear growth, we deduce that
\beqn
M_n(t) \le C_p\int_0^t ds  \left(\kappa_1(t-s)+\kappa_2(t-s)\right) M_{n-1}(s),\quad n\ge 0.
\eeqn
Moreover, from \eqref{5.6n},
\beq
\label{ch1'-s5.7-bis}
\sup_{s\in[0,T]} M_0(s) \le C_p\sup_{s\in[0,T]}\left( E(|u^1(s,y)|^p) + E(|u^0(s,y)|^p)\right) < \infty.
\eeq

Consider the sequence of functions defined for $t\in[0,T]$ by  $f_n(t):=M_n(t)$ and let $\kappa(t) = \kappa_1(t)+\kappa_2(t)$, $t\in[0,T]$. From \eqref{parabolic200} and \eqref{parabolic200-bis}, we see that $\int_0^T\kappa(s)\, ds <\infty$.
Then, applying Gronwall-type
Lemma \cite[Lemma C.1.3]{d-ss-2024} (with $z_0=0$ and $z\equiv 0$ there), we deduce that
\beq\label{rde4.1.10}
\sum_{n=0}^\infty\, \sup_{(t,x)\in [0,T] \times D}\left\Vert u^{n+1}(t,x)-u^n(t,x)\right\Vert_{L^p(\Omega)}
 < \infty.
\eeq
This implies that the sequence $\left(u^n(t,x),\, (t,x)\in [0,T] \times D\right)$, $n\ge 0$, converges in $L^p(\Omega)$, uniformly in $(t,x)\in [0,T] \times D$. That is, there exists $\left(u(t,x),\, (t,x)\in [0,T] \times D\right)$ such that
\beq\label{rde4.1.11}
   \lim_{n\to\infty}\, \sup_{(t,x)\in [0,T] \times D} \left\Vert u^{n}(t,x)-u(t,x)\right\Vert_{L^p(\Omega)} = 0.
\eeq
In particular, \eqref{5.302} holds.

For each $(t,x)$, $u^n(t,x)$ converges to $u(t,x)$ in probability (in fact, in $L^p(\Omega)$), so $(u(t,x))$ has a jointly measurable version by \cite[Lemma A.4.5]{d-ss-2024} which we again denote $(u(t,x))$.

For each $(t,x)$, $u(t,x)$ is $\cF_t$-measurable.
 Apply \cite[Lemma A.4.2 (a)]{d-ss-2024} with $(X,\cX) = (D, \B_D)$, and let $\cO$ denote the {\em optional} $\sigma$-field.  There is a $\cB_D \times \cO$-measurable function $(x,t,\omega) \mapsto \bar u(t,x,\omega)$ such that, for all $(t,x) \in [0,T]\times D, u(t,x) = \bar u(t,x)$ a.s. This modification $\bar u$ of $u$ is jointly measurable and adapted  since for all $t \in [0,T]$, $\cO\vert_{[0,t] \times \Omega} \subset \cB_{[0,t]} \times \F_t $. In the sequel, we use this modification and denote it $u$ instead of $\bar u$.


We have already seen that \eqref{5.302} is satisfied. Hence, the stochastic and pathwise integrals in \eqref{5.301} are well-defined.
\vskip 12pt


\noindent{\em Step 3.} We show that the stochastic process $\left(u(t,x),\, (t,x)\in [0,T] \times D\right)$ obtained in Step 2 satisfies equation \eqref{5.301}. Indeed, let
\begin{align}
\mathcal {I}(t,x) &= \int_0^t \int_D \Gamma(t,x;s,y) \sigma(s,y,u(s,y))\, W(ds,dy), \label{i-si}\\
\mathcal{J}(t,x) &=  \int_0^t ds \int_D dy\, \Gamma(t,x;s,y) b(s,y,u(s,y)). \label{j-pi}
\end{align}
Proceeding as in the proof of  \eqref{3-11} and \eqref{3-13}, but using the uniform Lipschitz continuity property of  $(\bf{H_L})$ instead of the linear growth, we obtain
\begin{align}
\label{ch1'-s5.10}
E\left(\left\vert\mathcal {I}^n(t,x) -\mathcal {I}(t,x)\right\vert^p\right) 
&\le C_p \left(\int_0^t ds\, \kappa_1(t-s)\right)^{\frac{p}{2}}\notag\\
&\quad\times \sup_{(s,x)\in [0,t] \times D}E\left(\left\vert u^n(s,x)-u(s,x)\right\vert^p\right),
\end{align}
and
\begin{align}
\label{ch1'-s5.11}
E\left(\left\vert\mathcal {J}^n(t,x) -\mathcal {J}(t,x)\right\vert^p\right) 
&\le \bar C_p \left(\int_0^t ds\, \kappa_2(t-s)\right)^p\notag\\
&\quad\times \sup_{(s,x)\in [0,t] \times D}E\left(\left\vert u^n(s,x)-u(s,x)\right\vert^p\right).
\end{align}
We have seen in \eqref{rde4.1.11} that the last right-hand sides of \eqref{ch1'-s5.10} and \eqref{ch1'-s5.11} converge to $0$ as $n\to\infty$.

With the notation introduced in Step 1, we have that
\beqn
u^{n+1}(t,x) = I_0(t,x) + \mathcal {I}^n(t,x) + \mathcal {J}^n(t,x).
\eeqn

The left-hand side converges to $u(t,x)$ in $L^p(\Omega)$, uniformly in $(t,x)\in [0,T] \times D$, while from \eqref{ch1'-s5.10} and \eqref{ch1'-s5.11}, the right-hand side converges to $I_0(t,x) + \mathcal {I}(t,x) + \mathcal {J}(t,x)$. Therefore, for each $(t,x)\in [0,T] \times D$,
\beqn
   u(t,x) = I_0(t,x) + \cI(t,x) + \cJ(t,x) \qquad \text{a.s.},
\eeqn
that is, equation \eqref{5.301} holds.
\vskip 12pt

\noindent{\em Step 4: Uniqueness}. Let
\beqn
\left(u(t,x),\, (t,x)\in [0,T] \times D\right),\quad \left(\tilde u(t,x), \, (t,x)\in [0,T] \times D\right),
\eeqn
 be two random field solutions to \eqref{5.301} satisfying \eqref{5.302}
with $p=2$. Using the same arguments as in \eqref{ch1'-s5.10}, \eqref{ch1'-s5.11}, we obtain
\begin{align*}
\sup_{x\in D} E\left(\left(u(t,x)-\tilde u(t,x)\right)^2\right)& \le C_2 \int_0^t ds\ (\kappa_1(t-s)+\kappa_2(t-s)) \\
&\quad\qquad\times\sup_{y\in D} E\left(\left\vert u(s,y)-\tilde u(s,y)\right\vert^2\right).
\end{align*}
Applying \cite[Lemma C.1.3, Equation (C.1.15)]{d-ss-2024} to the constant sequence
$
f(t) = f_n(t) := \sup_{x\in D} E\left(\left( u(t,x)-\tilde u(t,x)\right)^2\right),
$
 with $J(t):= \kappa(t)$, $z_0 = 0$ and $z\equiv 0$ there, we see that
 \beqn
 \sup_{x\in D} E\left(\left( u(t,x)-\tilde u(t,x)\right)^2\right) = 0, \quad {\text {for all}}\ t\in[0,T],
 \eeqn
 therefore $\tilde u$ is a version of $u$.
\end{proof}
\begin{remark}
\label{r-3-1}
Let $\Lambda=\delta_{0}(dx)$. In this case, the spectral measure is $\mu(dx)=dx$ and the Gaussian noise introduced at the beginning of Section \ref{s2} is a space-time white noise. The conditions $(a)$ and $(b)$ of Assumption $({\bf H_\Gamma})$ imply conditions $(ii)$, $(iiia)$, $(iiib)$ of \cite[Assumption $({\bf H_\Gamma})$, Chapter 4]{d-ss-2024}. Thus, \cite[Theorem 4.2.1]{d-ss-2024} implies Theorem \ref{ch5-t2} in this special case.
\end{remark}


\section{Sample path properties of parabolic SPDEs with non constant coefficients}
\label{ch6-s1}
\vskip 12pt

In this section, we consider the SPDE \eqref{5.1} driven by the second order partial differential operator in \eqref{5.300} satisfying condition \eqref{parabolic201}. By the discussion in Section \ref{s2} following the formulation of $({\bf H_\Gamma})$, we know that in this case, $\Gamma$ is positive and satisfies \eqref{convol}, \eqref{gauss} and \eqref{boundgauss}. Moreover, by \eqref{gauss-simple}, property  \eqref{parabolic200} holds if 
\beq
\label{dalang-cond}
\int_{\rek} \frac{\mu(d\xi)}{1+|\xi|^2} < \infty
\eeq
(see e.g. \cite{dalang-1999}).  In particular, $({\bf H_\Gamma})$ holds. In order to obtain sample path regularity properties of the solution to \eqref{5.1}, we introduce a modified version of $({\bf H_\Gamma})$:
\medskip

\noindent $({\bf \bar{H}_\Gamma})$ The functions $a_{i,j}(t,x)$, $b_i(t,x)$, $c(t,x)$ from $[0,T]\times D$ into $\re$ are $\cB_{[0,T]}\times \cB_D$-measurable and such that the following conditions hold:
\begin{description}
\item{(1)} The $a_{i,j}(t,x)$, $1\le i,j\le k$, satisfy \eqref{parabolic201}.
\item{(2) }The fundamental solution (or the Green's function) of $\cL$ is a function $\Gamma(t,x;s,y)$ from $\{(t,x,s,y)\in [0,T]\times D\times [0,T]\times D: 0\le s<t\le T\}$ into $\re$ that is jointly Borel measurable, that is, $\cB_{[0,T]}\times \cB_{D}\times \cB_{[0,T]}\times \cB_{D}$-measurable.
\item{(3)} The spectral measure $\mu$ is such that there exists $\eta\in\, ]0,1[$ such that
\beq
\label{5.311}
\int_{\rek} \frac{\mu(d\xi)}{(1+|\xi|^2)^\eta} < \infty.
\eeq
\end{description}
Observe that \eqref{5.311} is stronger than \eqref{dalang-cond} and that  $({\bf \bar{H}_\Gamma})$ implies $({\bf H_\Gamma})$ with $S(t,x)$ defined in \eqref{relevant}.

The main result of this section is the following theorem on sample path regularity of the random field solutions to \eqref{5.1}.

\begin{thm}
 \label{ch5-t30}
We assume that $({\bf H_\Lambda})$, $({\bf \bar{H}_\Gamma})$, $({\bf H_L})$ and $({\bf H_I})$ hold.
Let $(u(t, x), \, (t,x)\in[0,T]\times D)$ be the solution of \eqref{5.1} given in Theorem \ref{ch5-t2}. Then
for every $p\in[2,\infty[$, $\gamma_1\in \left]0,\frac{1-\eta}{2}\right [$ and $\gamma_2\in\, ]0,1-\eta[$, there exists a constant $C>0$ such that for all $s,t\in[0,T]$, $x,y\in D$,
\beq
 \label{5.312}
 E\left(|u(t,x) - u(s,y)|^p\right) \le C\left(|t-s|^{\gamma_1 p} + |x-y|^{\gamma_2 p}\right) + |I_0(t,x)-I_0(s,y)|^p.
\eeq
As a consequence, for $\gamma_1 \in \left]0,\frac{1-\eta}{2}\right [$ and $\gamma_2\in\, ]0,1-\eta[$, there is a version of $(u(t, x) - I_0(t, x), \, (t, x) \in [0, T] \times D)$ with  locally H\"older continuous sample paths with exponents $(\gamma_1, \gamma_2)$.
\end{thm}

\begin{remark}
\label{rd-4-r1}
In the case of the Riesz kernel $\Lambda(dx) = \vert x \vert^{-\beta}$, $\beta \in \, ]0, 2 \wedge k[$, and $\mu(d\xi) = \vert \xi \vert^{\beta - k}\, d\xi$, then condition \eqref{5.311} is satisfied if and only if $\eta > \beta/2$. Therefore, Theorem \ref{ch5-t30} states that for
$$\gamma_1 \in \left]0,\frac{2-\beta}{4}\right [\quad \text{and}\quad \gamma_2\in\, \left]0,\frac{2-\beta}{2}\right[,
$$
there is a version of $(u(t, x), \, (t, x) \in [0, T] \times D)$ with  locally H\"older continuous sample paths with exponents $(\gamma_1, \gamma_2)$. These intervals for $\gamma_1$ and $\gamma_2$ are known to be optimal: see \cite[Proposition 2.1]{DKN2013}.

\end{remark}


\medskip

\noindent{\em The factorization method}
\medskip

In \cite{dp-k-z-1987}, Da Prato, Kwapie\'n and Zabczyk introduced a technique known as the {\em factorization method}. Initially, it consisted of a factorization of the semigroup generated by the infinitesimal operator that defined certain evolution equations. In the context of this article, in the case $D = \R^k$, it will allow us to express the stochastic integral \eqref{si-ex-1} as a space-time convolution of a modified fundamental solution $\tilde \Gamma$ with another stochastic integral similar to \eqref{si-ex-1}. Often, this makes it possible to obtain (essentially) optimal space-time  sample path regularity properties of the stochastic integral. 

The factorization method will be used in the proof of Theorem \ref{ch5-t30}. In the next statement, we give the fundamental formula of \cite{dp-k-z-1987}.

\begin{prop}
\label{ch5-p30}
Let $\Gamma$ satisfy $({\bf H_\Lambda})$ and $({\bf \bar{H}_\Gamma})$.
 Let
$Z = (Z(t,x),\, (t,x)\in[0,T]\times D)
$ be a jointly measurable adapted process such that
\beqn
\sup_{(t,x)\in[0,T]\times D}E(|Z(t,x)|^p) < \infty.
\eeqn
Fix $\delta\in\left]0,\frac{1-\eta}{2}\right[$ and for $(t,x)\in [0,T]\times D$, define
\beq
\label{5.313}
Y_\delta(t,x) = \int_0^t \int_{D} \Gamma(t,x;s,y) (t-s)^{-\delta}Z(s,y) W(ds,dy).
\eeq
Then the stochastic process $(Y_\delta(t,x),\, (t,x)\in [0,T]\times D)$ is a well-defined jointly measurable adapted process such that, for any $p\in[2,\infty[$,
\beq
\label{5.3130}
\sup_{\txind} E\left(\left\vert Y_\delta(t,x)\right\vert^p\right) < \infty.
\eeq
In addition, the following formula holds:
\begin{align}
\label{5.314}
&\int_0^t \int_{D} \Gamma(t,x;s,y)Z(s,y) W(ds,dy)\nonumber\\
&\qquad  = \frac{\sin(\pi\delta)}{\pi} \int_0^t ds\, (t-s)^{\delta-1}\int_{D} dy\,  \Gamma(t,x;s,y)\, Y_\delta(s,y).
\end{align}
\end{prop}
\begin{proof}
First we prove that the process $(Y_\delta(t,x), (t,x)\in [0,T]\times D)$ is well-defined. For this, we have to check that
the stochastic process
\beqn
  \left (\Gamma(t,x;s,y) (t-s)^{-\delta}Z(s,y),\, (s,y)\in[0,T]\times D \right)
\eeqn
 belongs to $L^2(\Omega; U_T)$. 

Fix $0<s<t\le T$, $x\in D$. As in the computations for \eqref{comp-norm}, we obtain
\begin{align*}
&E\left(\Vert\Gamma(t,x;\cdot,\ast) (t-\cdot)^{-\delta}Z(\cdot,\ast)1_D(\ast)\Vert_{U_T}^2\right)\\
& \quad= E\left(\int_0^t ds\, (t-s)^{-2\delta}\right.\\
&\left.\qquad\qquad\times\int_{D}dy \int_{D} dz\,  \Gamma(t,x;s,y)Z(s,y)f(y-z)\Gamma(t,x;s,z)Z(s,z)\right)\\
 &\quad\le E\left(\int_0^t ds\, (t-s)^{-2\delta}\right.\\
 &\left.\qquad\qquad\times\int_{D}dy \int_{D} dz\,  \Gamma(t,x;s,y)\, |Z(s,y)|\, f(y-z)\, \Gamma(t,x;s,z) \, |Z(s,z)|\, \right)\\
&\quad\le 
  C_Z \,
 \int_0^t ds\,  (t-s)^{-2\delta} \int_{\rek}dy \int_{\rek} dz\,  S(t-s,x-y)f(y-z)S(t-s,x-z)\\
& \quad \le C_Z   \int_0^t ds\,  (t-s)^{-2\delta}\int_{\rek}\mu(d\xi)\, |\tf S(t-s)(\xi)|^2,
 \end{align*}
 where $C_Z:=  \sup_{(t,x)\in[0,T]\times D} E\left(\vert Z(t,x)\vert^2\right)$.

Since $\cF S(s, *)(\xi) = \exp(-\half c s\, \vert \xi\vert^2)$, by applying Fubini's theorem and then integrating with respect to the time variable, we obtain
\begin{align*}
&\int_0^t ds\, (t-s)^{-2\delta}\int_{\rek}\mu(d\xi)\, |\tf S(t-s)(\xi)|^2\\
&\qquad = \int_0^t ds\, (t-s)^{-2\delta}\int_{\rek}\mu(d\xi)\, \exp(-c(t-s)|\xi|^2)\\
&\qquad = \int_{\rek}\mu(d\xi) \int_0^t  ds\,  s ^{-2\delta}\exp(-c s\, |\xi|^2)\\
&\qquad =  c^{2 \delta - 1} \int_{\rek}\mu(d\xi)\ |\xi|^{4\delta-2} \int_0^{ct|\xi|^2} dr\, r^{-2\delta} \exp(-r),
\end{align*}
where, in the last equality, we have applied the change of variable $r= c s|\xi|^2$.

If  $|\xi|\le 1$, then for $t \in [0, T]$, $|\xi|^{4\delta-2} \int_0^{t|\xi|} dr\, r^{-2\delta} \exp(-r)\le C$. Moreover, if $|\xi|>1$, then
\beqn
|\xi|^{4\delta-2} \int_0^{c t|\xi|} dr\, r^{-2\delta} \exp(-r) \le \Gamma_E(1-2\delta) (1+|\xi|^2)^{2\delta-1},
\eeqn
where $\Gamma_E$ again denotes the Euler Gamma function.

Consequently,
\beqn
\int_0^t ds\, (t-s)^{-2\delta}\int_{\rek}\mu(d\xi)\, |\tf S(t-s)(\xi)|^2 \le C \int_{\rek} \frac{\mu(d\xi)}{(1+|\xi|^2)^{1-2\delta}} < \infty,
\eeqn
where we have used assumption \eqref{5.311} and the fact that $\delta\in\left]0,\frac{1-\eta}{2}\right[$.
This proves that  $E\left(\Vert\Gamma(t,x;\cdot,\ast) (t-\cdot)^{-\delta}Z(\cdot,\ast) 1_D(\ast)\Vert_{U_T}^2\right) < \infty$. 

The joint measurability of the process $(Y_\delta(t,x),\, (t,x)\in[0,T]\times \rek)$ follows from an extension of Proposition 2.6.2 in \cite{d-ss-2024}.
Let us now prove \eqref{5.3130}. Burkholder's inequality 
yields
\beq
\label{5.3131}
E\left(\left\vert Y_\delta(t,x)\right\vert^p\right) \le CE\left(\int_0^t ds\,  (t-s)^{-2\delta}\Vert \Gamma(t,x;s,*) Z(s,*)1_D\Vert^2_U \right)^{\frac{p}{2}}.
\eeq
For any $\delta\in\left]0,\frac{1}{2}\right[$, the function $s\mapsto (t-s)^{-2\delta}$ is integrable on $[0,t]$. Hence, by applying H\"older's inequality, the right-hand side of \eqref{5.3131} is bounded above by
\beqn
C\left( \int_0^t  ds\, (t-s)^{-2\delta}\right)^{\frac{p}{2}-1} \int_0^t  ds\, (t-s)^{-2\delta} E\left(\Vert \Gamma(t,x;s,*) Z(s,*)1_D\Vert^p_U\right).
\eeqn
Because of the properties of $\Gamma$ and $Z$, we can apply Proposition \ref{ch2-incluU} and Lemma \ref{l-1} to obtain 
\begin{align*}
&E\left(\Vert \Gamma(t,x;s,*) Z(s,*)1_D\Vert^p_U\right)\\
&\quad \le \sup_{\txind} E(|Z(t,x)|^p)\\
&\quad\qquad \qquad\times \left(\int_{D} dy  \int_{D} dz\, G(t-s,x-y) f(y-z) G(t-s,x-z)\right)^{\frac{p}{2}}\\
&\quad \le C  \left(\int_{\rek}  \mu(d\xi)\, |\tf G(t-s)(\xi)|^2\right)^{\frac{p}{2}}.
\end{align*}
Thus,
\beqn
E\left(\left\vert Y_\delta(t,x)\right\vert^p\right) \le C\left(\int_0^t  ds \, (t-s)^{-2\delta}\right)^{\frac{p}{2}},
\eeqn
proving \eqref{5.3130}.

Formula \eqref{5.314} follows from \eqref{5.313}, \eqref{convol}, an extension of the stochastic Fubini's theorem \cite[Theorem 2.4.1]{d-ss-2024} (or \cite[Theorem 4.18]{dz}), and the formula
\beqn
\int_0^t ds\, (t-s)^{\delta-1} s^{-\delta} = \frac{\pi}{\sin(\pi\delta)},
\eeqn
which is obtained by using the Beta function. Indeed,
\begin{align*}
& \int_0^t ds \int_{D} dy\, \Gamma(t,x;s,y) (t-s)^{\delta-1}Y_\delta(s,y)\\
&\quad = \int_0^t ds  \int_{D} dy\, \Gamma(t,x;s,y) (t-s)^{\delta-1}\\
&\qquad \qquad\times\left( \int_0^s \int_{D} \Gamma(s,y;r,z) (s-r)^{-\delta}Z(r,z)\, W(dr,dz) \right)\\
&\quad = \int_0^t  \int_{D}  W(dr,dz)\, Z(r,z) \left(\int_r^t ds\,  (t-s)^{\delta-1} (s-r)^{-\delta} \right)\\
&\qquad\qquad\times\left(\int_D dy\,  \Gamma(t,x;s,y) \Gamma(s,y;r,z)\right)
\end{align*}
Since
\beqn
\int_r^t ds\,  (t-s)^{\delta-1} (s-r)^{-\delta} =\int_0^{t-r} d\rho\,  (t-r-\rho)^{\delta-1} \rho^{-\delta}= B(1-\delta,\delta)= \frac{\pi}{\sin(\pi\delta)},
\eeqn
where $B(\alpha,\beta)$ denotes the Beta function, the last expression in the array is equal to
\beqn
\left(\frac{\sin(\pi\delta)}{\pi}\right)^{-1}\int_0^t  \int_{D}  Z(r,z) \Gamma (t,x;r,z)\, W(dr,dz),
\eeqn
where we have used \eqref{convol}
\end{proof}
\bigskip

\begin{remark}
\label{ch5-r2}
In the definition of the stochastic integral in \eqref{5.313}, the singularity of the function $s\mapsto\Gamma(t,x;s,y)(t-s)^{-\delta}$ on the diagonal $s=t$ is stronger than that of $s\mapsto \Gamma(t,x;s,y)$. The role of Hypothesis \eqref{5.311} is precisely to compensate this fact. Indeed, this condition means that the noise $W$ is sufficiently smooth to integrate this stronger singularity.
\end{remark}
\vskip 12pt

\noindent{\em Towards the proof of Theorem \ref{ch5-t30}}
\medskip

The next several assertions contain the necessary ingredients for the proof of Theorem \ref{ch5-t30}.
\vskip 12pt

\noindent{\em Increments of the pathwise integral in \eqref{5.301}}
\medskip

\begin{lemma}
\label{ch5-p30}
(Spatial increments) The hypotheses are $({\bf H_\Lambda})$, $({\bf \bar{H}_\Gamma})(1)$-$(2)$, \eqref{dalang-cond}, $({\bf H_L})$ and $({\bf H_I})$, and $(u(t,x),\, \txind)$ is defined in \eqref{5.301}. Then for all
$p\in [1,\infty[$, there is $C_p < \infty$ such that for all $x_1, x_2\in D$,
\begin{align}
\label{5.315}
&\sup_{t\in[0,T]} E\left(\left\vert \int_0^t ds  \int_{D} dy\, [\Gamma(t,x_1;s,y) - \Gamma(t,x_2;s,y)]b(s, y,u(s,y))\right\vert^p\right)\notag\\
&\qquad \le C_p\, |x_1-x_2|^p.
\end{align}
\end{lemma}

\begin{proof}
By H\"older's inequality, and using the linear growth of $b$ and \eqref{5.302}, the proof of \eqref{5.315} reduces to checking that
\beq
\label{5.316}
\int_0^t ds  \int_{D} dy \left\vert \Gamma(t,x_1;s,y) - \Gamma(t,x_2;s,y)\right\vert
\le C |x_1-x_2|.
\eeq
Let $\varphi(\lambda) = \Gamma(t,x_1+\lambda(x_2-x_1); s,y)$. We have
\begin{align*}
&\int_0^t ds  \int_{D} dy \left\vert \Gamma(t,x_1;s,y) - \Gamma(t,x_2;s,y)\right\vert
=\int_0^t ds  \int_{D} dy \left\vert \int_0^1 d\lambda\,  \varphi^{\prime}(\lambda)\right\vert\\
& \quad \le C |x_1-x_2| \int_0^t ds   \int_{D} dy \int_0^1 d\lambda\,  (t-s)^{-\frac{k+1}{2}}\\
&\qquad\qquad\qquad\qquad\qquad\times \exp\left(-c\, \frac{|x_1+\lambda(x_2-x_1)-y|^2}{t-s}\right)\\
& \quad \le C |x_1-x_2| \int_0^t ds\,  (t-s)^{-\frac{1}{2}}\le C   |x_1-x_2| .
\end{align*}
The first inequality above is obtained by applying \eqref{gauss}.  For the second inequality, we use  Fubini's theorem and integrate the Gaussian density.
\end{proof}


For the study of the time increments of the pathwise integral, we will make use of the next lemma.

\begin{lemma}
\label{ch6-l1}
Suppose that assumptions $({\bf \bar{H}}_\Gamma)(1)$-$(2)$ hold. Then: 
\begin{description}
\item{(a)} For any $\delta\in\, ]0,1[$, there is a constant $C_\delta < \infty$ such that, for  $0<t\le t+h < \infty$ and $x \in \rek$, 
\beq
\label{6.1}
\int_0^t ds\ (t-s)^{\delta-1} \int_D dy\ |\Gamma(t+h,x;s,y) - \Gamma(t,x;s,y)| \le C_\delta\, h^\delta.
\eeq
 \item{(b)} There is $C < \infty$ such that, for any $\alpha \in\, ]0,1[$, there is $c_\alpha < \infty$ such that $\lim_{\alpha \uparrow 1} c_\alpha = 1$ and for any $0<t\leq t+ h < \infty$ and $x \in \rek$,
\beq
\label{6.1a}
	  \int_0^t ds \int_D dy\ |\Gamma(t+h,x;s,y) - \Gamma(t,x;s,y)| \le C\, c_\alpha\, \frac{t^{1-\alpha}}{1-\alpha}\, h^\alpha.
\eeq
\end{description}
\end{lemma}

\begin{proof}
(a) Clearly,
\begin{align*}
&\int_0^t ds\, (t-s)^{\delta-1} \int_D dy\ |\Gamma(t+h,x;s,y) - \Gamma(t,x;s,y)|\\
&\qquad = \int_0^t ds\, (t-s)^{\delta-1} \int_D dy\ \left\vert \int_t^{t+h} d\tau\, \partial_{\tau} \Gamma(\tau,x;s,y)\right\vert.
\end{align*}
Apply \eqref{gauss} to see that the last expression is bounded by
\begin{align*}
& C \int_0^t ds\, (t-s)^{\delta-1} \int_D dy \int_t^{t+h} d\tau\, (\tau - s)^{-\frac{k+2}{2}} \exp\left(-c\,\frac{|x-y|^2}{\tau - s}\right)\\
&\qquad = C \int_0^t ds\ (t-s)^{\delta-1} \int_t^{t+h} d\tau\, (\tau - s)^{-1},
\end{align*}
where we have applied Fubini's theorem and integrated the Gaussian density.

The last integral can be bounded as follows. First, we integrate with respect to $\tau$ to obtain
\begin{align}\notag
&\int_0^t ds\, (t-s)^{\delta-1} \int_t^{t+h} d\tau\, (\tau - s)^{-1} \\ \notag
   &\qquad = \int_0^t ds\, (t-s)^{\delta-1}\log \left(1+\frac{h}{t-s}\right)
 = \int_0^t ds\, s^{\delta-1}\log \left(1+\frac{h}{s}\right)\\ \notag
&\qquad  = h^\delta \int_0^{\frac{t}{h}} dr\,  r^{\delta-1}\log \left(1+\frac{1}{r}\right)\\ &\qquad  \le C_\delta\, h^\delta,
\label{6.1b}
\end{align}
where the last equality is obtained by the change of variables $s=rh$, and
\beqn
   C_\delta = \int_0^{+\infty} dr\, r^{\delta - 1} \log \left(1+\frac{1}{r}\right).
\eeqn
Notice that $C_\delta < \infty$ for $\delta \in \, ]0,1[$ (but $\lim_{\delta \uparrow 1} C_\delta = + \infty$ because of the behavior of the integrand at $+\infty$).

   (b) Using the same calculations as in part (a) up to \eqref{6.1b}, but with $\delta = 1$, we see that the left-hand side of \eqref{6.1a} is bounded above by
\beqn
   C \int_0^t ds\ \log \left(1+\frac{h}{s}\right).
\eeqn
Use the inequality $\log(1+x) \le c_\alpha\, x^\alpha$, valid for all $x \ge 0$ and $\alpha \in\, ]0,1]$ (with $\lim_{\alpha \uparrow 1} c_\alpha = 1$), to see that for $0 < \alpha < 1$, this is bounded above by
\beqn
   C\, c_\alpha \, h^\alpha \int_0^t ds\, s^{-\alpha} = C\, c_\alpha \, \frac{t^{1-\alpha}}{1-\alpha} \, h^\alpha .
\eeqn
This proves the lemma.
\end{proof}

\begin{lemma}
\label{ch5-p31}
(Increments in time) Assume $({\bf H_\Lambda})$, $({\bf \bar{H}_\Gamma})(1)$-$(2)$, \eqref{dalang-cond}, $({\bf H_L})$ and $({\bf H_I})$.
Let $(u(t,x), \txind)$ be as defined in \eqref{5.301} and
fix $\gamma \in \, ]0,1[$.
Then for all $p\in [1,\infty[$, there is a constant $C_{p,\gamma, T} < \infty$ such that, for all $t_1,t_2 \in [0,T]$ and $x \in D$, 
\begin{align}
\label{5.317}
&\sup_{x\in D}E\left(\left\vert \int_0^{t_1} ds  \int_{D} dy\, \Gamma(t_1,x;s,y)b(s, y,u(s,y))\right.\right.\notag\\
&\left.\left. \qquad\qquad  -\int_0^{t_2} ds  \int_{D} dy\, \Gamma(t_2,x;s,y)b(s, y,u(s,y))\right\vert^p\right)\notag\\
&\qquad\le C_{p,\gamma, T}\, |t_1-t_2|^{\gamma p}.
\end{align}
\end{lemma}

\begin{proof}
To ease the notation, we write $t_1 = t+h$, $t_2 = t$, where $t\in[0,T]$ and $h>0$ is such that $t+h\le T$. Notice that
\begin{align*}
  & E\left(\left\vert \int_0^{t+h} ds  \int_{D} dy \left[ \Gamma(t+h,x;s,y)- \Gamma(t,x;s,y) \right] b(s, y,u(s,y))\right\vert^p\right) \\
  &= A_1 + A_2,
\end{align*}
where
\begin{align*}
   A_1 &= E\left(\left\vert \int_t^{t+h} ds  \int_{D} dy \, \Gamma(t+h,x;s,y)b(s, y,u(s,y))\right\vert^p\right), \\
   A_2 &= E\left(\left\vert \int_0^{t} ds  \int_{D} dy \left[ \Gamma(t+h,x;s,y)- \Gamma(t,x;s,y) \right] b(s, y,u(s,y))\right\vert^p\right).
\end{align*}
First, we consider
the term $A_1$,
to which we apply H\"older's inequality with respect to the finite measure $ \Gamma(t+h,x;s,y) ds dy$. By the properties of $b$ and  \eqref{5.302}, we obtain
\begin{align}
\label{5.318}
&E\left(\left\vert \int_t^{t+h} ds  \int_{D} dy\,  \Gamma(t+h,x;s,y)b(s, y,u(s,y))\right\vert^p\right)\notag\\
&\qquad \le C_p\left(1+\sup_{\txind}E(|u(t,x)|^p)\right) \left(\int_t^{t+h} ds  \int_{D} dy\,  \Gamma(t+h,x;s,y)\right)^p\notag\\
&\qquad = \tilde C_p\, h^p,
\end{align}
for all in $\txind$.

Next, we consider the term $A_2$.
As before, we use H\"older's inequality and the linear growth of $b$ to see that it is bounded from above by
\begin{align}
\label{5.319}
&C_p\left(1+\sup_{\txind}E(|u(t,x)|^p)\right)\notag\\
&\qquad \times \left(\int_0^t ds \int_{D} dy\, |\Gamma(t+h,x;s,y)- \Gamma(t,x;s,y)|\right)^p\notag\\
&\quad \le C_{p,\gamma, T}\, h^{\gamma p},
\end{align}
where the last inequality follows from \eqref{5.302} and Lemma \ref{ch6-l1}(b).

From \eqref{5.318} and \eqref{5.319}, we obtain \eqref{5.317}.
\end{proof}

Lemmas \ref{ch5-p30} and \ref{ch5-p31} yield the following.

\begin{prop}
\label{prop4.9rd}
Assume $({\bf H_\Lambda})$, $({\bf \bar{H}_\Gamma})(1)$-$(2)$, \eqref{dalang-cond}, $({\bf H_L})$ and $({\bf H_I})$. Let the process $(\cJ(t,x))$ be as defined in \eqref{j-pi}.
For any $p\in [1,\infty[$ and $\gamma \in \, ]0,1[$, for any $(t,x), (s,y)\in [0,T]\times D$,
\beq
\label{5.321}
E(|\cJ(t,x) - \cJ(s,y)|^p) \le C (|t-s|^{\gamma p} + |x-y|^{p}).
\eeq
\end{prop}
\vskip 12pt

\noindent{\em Increments of the stochastic integral in \eqref{5.301}}
\medskip

\begin{lemma}
\label{ch5-p32}
(Spatial increments)
The hypotheses are as in Theorem \ref{ch5-t30}, and $(u(t,x),\, \txind)$ is defined in \eqref{5.301}.
Then for every  $\gamma\in\  ]0,1-\eta[$ and all $p\in [2,\infty[$, there is $C>0$ such that for all $x_1, x_2\in D$,  
\begin{align}
\label{5.322}
&\sup_{t\in[0,T]} E\left(\left\vert \int_0^t   \int_{D}  [\Gamma(t,x_1;s,y) - \Gamma(t,x_2;s,y)]\, \sigma(s, y,u(s,y))\, W(ds,dy)\right\vert^p\right)\notag\\
 &\qquad \qquad\le C |x_1-x_2|^{\gamma p}.
\end{align}
\end{lemma}

\begin{proof}
Fix $\gamma\in\, ]0,1-\eta[$ and $\delta\in \left] \frac{\gamma}{2},\frac{1-\eta}{2}\right[$. We apply the identity \eqref{5.314}, H\"older's inequality and then \eqref{5.3130} to obtain
\begin{align}
\label{5.323}
&E\left(\left\vert \int_0^t   \int_{D}  [\Gamma(t,x_1;s,y) - \Gamma(t,x_2;s,y)]\, \sigma(s, y,u(s,y))\, W(ds,dy)\right\vert^p\right)\notag\\
&\qquad \le C \left(\int_0^t  ds (t-s)^{\delta-1} \int_D dy\,  \vert \Gamma(t,x_1;s,y) - \Gamma(t,x_2;s,y)\vert\right)^p.
\end{align}
Clearly,
\begin{align*}
&\int_0^t ds\, (t-s)^{\delta-1} \int_D dy\, \vert \Gamma(t,x_1;s,y) - \Gamma(t,x_2;s,y)\vert\\
&\quad =  \int_0^t ds\,  (t-s)^{\delta-1}  \left(\int_D dy \, |\Gamma(t,x_1;s,y) -  \Gamma(t,x_2;s,y)|\right)^{1-\gamma}\\
& \qquad\qquad \times\left(\int_D dy \, |\Gamma(t,x_1;s,y) -  \Gamma(t,x_2;s,y)|\right)^\gamma.
\end{align*}
By \eqref{boundgauss},
\begin{align*}
& \left(\int_D dy \,|\Gamma(t,x_1;s,y) - \Gamma(t,x_2;s,y)|\right)^{1-\gamma}\\
&\quad \le  \left(\int_D dy\,|\Gamma(t,x_1;s,y)| + \int_D dy\, |\Gamma(t,x_2;s,y)|\right)^{1-\gamma} < \infty.
\end{align*}
Hence,
\begin{align*}
&\int_0^t  (t-s)^{\delta-1} \int_D dy\, \vert \Gamma(t,x_1;s,y) - \Gamma(t,x_2;s,y)\vert\\
&\quad \le C \int_0^t ds\,  (t-s)^{\delta-1}  \left(\int_D dy\, |\Gamma(t,x_1;s,y) -  \Gamma(t,x_2;s,y)|\right)^\gamma.
 \end{align*}
Consider the function
$\varphi(\lambda) = \Gamma(t,x_1+\lambda(x_2-x_1); s,y)$, $\lambda\in\, ]0,1[$,
introduced in the proof of Lemma \ref{ch5-p30}. Then
\begin{align*}
&\int_0^t ds\, (t-s)^{\delta-1} \int_D dy\,  \vert \Gamma(t,x_1;s,y) - \Gamma(t,x_2;s,y)\vert\\
&\qquad \le C \int_0^t ds \, (t-s)^{\delta-1}  \left(\int_D dy\, |\Gamma(t,x_1;s,y) -  \Gamma(t,x_2;s,y)|\right)^\gamma\\
&\qquad \le C |x_1-x_2|^\gamma\int_0^t ds\,  (t-s)^{\delta-1}\\
&\qquad\qquad\times \left[\int_D dy  \int_0^1 d\lambda\,  (t-s)^{-\frac{k+1}{2}}\exp\left(-c\,\frac{|x_1+\lambda(x_2-x_1)-y|^2}{t-s}\right)\right]^\gamma\\
 &\qquad\le C  |x_1-x_2|^\gamma\int_0^t ds\,  (t-s)^{\delta-1-\frac{\gamma}{2}}\\
 &\qquad\qquad\times  \left[\int_0^1 d\lambda \ \int_D dy\,  (t-s)^{-\frac{k}{2}}\exp\left(-c\,\frac{|x_1+\lambda(x_2-x_1)-y|^2}{t-s}\right)\right]^\gamma,
 \end{align*}
 where in the last step, we have used Fubini's theorem.
 By integrating the Gaussian density, we see that the integral within brackets is finite. As for the remaining integral, it is finite since
 $\delta-1-\frac{\gamma}{2}>-1$. This ends the proof of Lemma \ref{ch5-p32}.
  \end{proof}

\begin{lemma}
\label{ch5-p33}
(Time increments) The hypotheses are as in Theorem \ref{ch5-t30}, and $(u(t,x),\, \txind)$ is defined in \eqref{5.301}.
Then, for all  $\gamma\in \left]0,\frac{1-\eta}{2}\right[$ and all $p\in [2,\infty[$, there is a positive constant $C$ such that, for every $t_1, t_2\in [0,T]$, 
\begin{align}
\label{5.324}
&\sup_{x\in D} E\left(\left\vert \int_0^{t_1}   \int_{D}  \Gamma(t_1,x;s,y) \sigma(s, y,u(s,y))\, W(ds,dy)\right.\right.\notag\\
&\left.\left.\qquad \qquad - \int_0^{t_2}   \int_{D}  \Gamma(t_2,x;s,y)]\sigma(s, y,u(s,y))\, W(ds,dy)\right\vert^p\right)\notag\\
 &\qquad \le C |t_1-t_2|^{\gamma p}.
 \end{align}
 \end{lemma}

 \begin{proof}
As in the proof of Lemma \ref{ch5-p31}, we write $t_1 = t+h$, $t_2 = t$, where $t\in[0,T]$ and $h>0$ is such that $t+h\le T$. Fix $\delta=\gamma\in \left]0,\frac{1-\eta}{2}\right[$ and use formula \eqref{5.314} and the triangle inequality to obtain
\begin{align*}
&E\left(\left\vert \int_0^{t+h}   \int_{D}  \Gamma(t+h,x;s,y) \sigma(s, y,u(s,y)) \, W(ds,dy)\right.\right.\notag\\
&\left.\left. \qquad \qquad - \int_0^{t}   \int_{D}  \Gamma(t,x;s,y)]\sigma(s, y,u(s,y))\, W(ds,dy)\right\vert^p\right)\notag\\
&\qquad \le I_1(t,h,x) + I_2(t,h,x),
\end{align*}
where
\begin{align*}
I_1(t,h,x) & = E\left(\left\vert \int_t^{t+h} ds\, (t+h-s)^{\delta-1} \int_D dy\, \Gamma(t+h,x;s,y) Y_\delta(s,y)\right\vert^p\right),\\
I_2(t,h,x) & = E\left(\left\vert \int_0^t ds\,  (t+h-s)^{\delta-1}  \int_D dy\, \Gamma(t+h,x;s,y) Y_\delta(s,y)\right.\right.\\
& \left.\left.\qquad  -\int_0^t ds\,  (t-s)^{\delta-1}  \int_D dy\, \Gamma(t,x;s,y)Y_\delta(s,y)\right\vert^p\right).
\end{align*}
 For the study of $I_1(t,h,x)$, we apply H\"older's inequality with respect to the measure on $[t,t+h]\times D$ given by
$(t+h-s)^{\delta-1} \Gamma(t+h,x;s,y)ds dy$, along with \eqref{5.3130}, to obtain
\begin{align}
I_1(t,h,x) & \le \sup_{\txind} E(|Y_\delta(t,x)|^p)\notag\\
&\qquad\times \left(\int_t^{t+h} ds\,  (t+h-s)^{\delta-1} \int_D dy\, \Gamma(t+h,x;s,y)\right)^p\notag\\
& \le C  \left(\int_t^{t+h} ds\,  (t+h-s)^{\delta-1} \right)^p \notag\\
& \le C h^{\delta p},
\label{5.325}
\end{align}
 where we have also used  \eqref{boundgauss}.

 Next, we study $I_2(t,h,x)$. As before, we start by applying H\"older's inequality and then the triangle inequality and \eqref{boundgauss}, to obtain
  \begin{align}\notag
 & I_2(t,h,x)\\
  &\qquad \le C\sup_{\txind} E(|Y_\delta(t,x)|^p)
 \left[\left(\int_0^t ds\,  |(t+h-s)^{\delta-1}-  (t-s)^{\delta-1}|\right)^p\right.\notag\\
 &\qquad\qquad\left. + \left(\int_0^t ds\,  (t+h-s)^{\delta-1} \int_D dy\, |\Gamma(t+h,x;s,y)-\Gamma(t,x;s,y)|\right)^p\right].
\label{5.326}
\end{align}
Since $\delta<1$, we have
\begin{align}
\label{5.327}
& \int_0^t ds\, |(t+h-s)^{\delta-1}-  (t-s)^{\delta-1}| \\  \notag
 & \qquad = \int_0^t ds\ ((t-s)^{\delta-1} - (t+h-s)^{\delta-1})
 = \int_0^t ds\ (s^{\delta-1} - (s+h)^{\delta-1}) \\
 &\qquad = \frac{1}{\delta}\, (t^\delta - (t+h)^\delta + h^\delta)\notag\\
& \qquad \le c\, h^\delta.
\end{align}
This provides a control on the term in the second line of \eqref{5.326}.

As for the term in the third line of \eqref{5.326}, we observe that since $\delta < 1$, it is bounded by
\beqn
\left(\int_0^t ds\,  (t-s)^{\delta-1} \int_D dy\  |\Gamma(t+h,x;s,y)-\Gamma(t,x;s,y)|\right)^p.
\eeqn
Thus, by applying Lemma \ref{ch6-l1} (a), we see that it is bounded above by $Ch^{\delta p}$.

The proposition is  proved.
\end{proof}
 Propositions \ref{ch5-p32} and \ref{ch5-p33} yield the following.

\begin{prop}
\label{rd17.10p2}
The hypotheses are as in Theorem \ref{ch5-t30}. Let the process $(\cI(t,x))$ be as defined in \eqref{i-si}.
For any $p\in [2,\infty[$, $\gamma_1\in \left]0, \frac{1-\eta}{2}\right[$ and $\gamma_2\in\  ]0, 1-\eta[$, there is $C< \infty$ such that for all $(t,x), (s,y) \in [0,T]\times D$,
\beq
 \label{5.325}
 E(|\cI(t,x)-\cI(t,y)|) \le C (|t-s|^{\gamma_1p} + |x-y|^{\gamma_2 p}).
\eeq
\end{prop}
\medskip

\noindent{\em Proof of Theorem \ref{ch5-t30}}. By \eqref{5.301},
$u(t, x) = I_0(t, x) + \cI(t, x) + \cJ(t, x)$.
Therefore, inequality \eqref{5.312} is a consequence of Propositions \ref{prop4.9rd} and \eqref{5.325}.

The statement concerning the existence of a H\"older continuous version of $(u(t, x) - I_0(t, x))$ is a consequence of the anisotropic Kolmogorov's continuity criterion \cite[Theorem A.2.1]{d-ss-2024}.
\hfill $\Box$

{\bf Author information:}
\vskip 16pt

{\scshape
Robert C.~Dalang}

Institut de Mathématiques

\'Ecole Polytechnique Fédérale de Lausanne (EPFL)

CH-1015 Lausanne

Switzerland
\smallskip 

robert.dalang@epfl.ch

\vskip 16pt

{\scshape
Marta Sanz-Sol\'e}

Facultat de Matemàtiques i Informàtica

Universitat de Barcelona

Gran Via de les Corts Catalanes 585

08007 Barcelona

Spain
\smallskip 

marta.sanz@ub.edu

\end{document}